\documentclass{amsart}
\usepackage{amsmath}
\usepackage{bbm}
\usepackage{latexsym}
\usepackage{mathrsfs}
\usepackage{graphicx}
\usepackage{xcolor}
\usepackage{hyperref}
\usepackage{verbatim} 
\usepackage{mathtools} 
\usepackage[mathcal]{euscript}

\newcommand{\R}{\mathbb{R}}
\newcommand{\N}{\mathbb{N}}

\newcommand{\wh}[1]{\widehat{#1}}

\newcommand{\mc}[1]{\mathcal{#1}}
\newcommand{\ov}[1]{\overline{#1}}

\newcommand{\Rank}[1]{\mathrm{rank}\,#1}

\newcommand{\Span}[1]{\mathrm{span}\,#1}

\newcommand{\Vc}{\mathrm{Vec}} 
\newcommand{\Pf}{\mathrm{Pf}}

\newcommand{\adjp}{\mathrm{adj}^\Pf}

\newcommand{\Vect}{\mathrm{Vec}}

\newtheorem{thm}{Theorem}

\newtheorem{lemma}[thm]{Lemma}
\newtheorem{cor}[thm]{Corollary}
\newtheorem{prop}[thm]{Proposition}
\newtheorem{corollary}[thm]{Corollary}

\theoremstyle{definition}
\newtheorem{defi}[thm]{Definition}

\theoremstyle{definition}

\theoremstyle{remark} 
\newtheorem{remark}[thm]{Remark}

\newcommand{\be}{\begin{equation}}
\newcommand{\ee}{\end{equation}}

\mathtoolsset{showonlyrefs,showmanualtags} 
\numberwithin{equation}{section}

\title[Fuller singularities for generic control-affine 
systems]{
Fuller singularities for generic control-affine 
systems with an even number of controls
}
\date{\today}



\thanks{This project has been supported by the ANR SRGI (reference ANR-15-CE40-0018) and by a public grant as part of the \emph{Investissement d'avenir project}, reference ANR-11-LABX-0056-LMH, LabEx LMH, in a joint call with \emph{Programme Gaspard Monge en Optimisation et Recherche Op\'erationnelle}. F.~B. is also supported by University of Padova STARS Project ``Sub-Riemannian Geometry and Geometric Measure Theory Issues: Old and
	New", and by GNAMPA of INdAM (Italy) through projects “Rectifiability in Carnot Groups”}

\author{Francesco Boarotto}
\address{Dipartimento di Matematica Tullio Levi-Civita,
 Universit\`a degli studi di Padova, Italy}
\email{\href{mailto:francesco.boarotto@math.unipd.it}{\nolinkurl{francesco.boarotto@math.unipd.it}}}

\author{Yacine Chitour}
\address{Universit\'e Paris-Sud, L2S, CentraleSup\'elec, Universit\'e Paris-Saclay, Gif-sur-Yvette, France}
\email{\href{mailto:yacine.chitour@l2s.centralesupelec.fr}{\nolinkurl{yacine.chitour@l2s.centralesupelec.fr}}}

\author{Mario Sigalotti}
\address{Inria 
\& Laboratoire Jacques-Louis Lions, CNRS, Sorbonne Universit\'e, Universit\'e de Paris,  
France}
\email{\href{mailto:Mario.Sigalotti@inria.fr}{\nolinkurl{Mario.Sigalotti@inria.fr}}}

\begin{document}
	
	\begin{abstract} In this article we study how bad can be the singularities of a time-optimal trajectory of a generic control affine system. 
	Under the assumption that  the control has an even number of scalar components and belongs to a closed ball
	we prove that singularities cannot be, generically, worse than finite order accumulations of Fuller points, with order of accumulation lower than a bound depending only on the dimension of the manifold where the system is set. 
	\end{abstract}
	
	\maketitle

\section{Introduction}\label{sec:intro}

\subsection{Time-optimal trajectories of control-affine systems}

Let $M$ be a smooth and connected $n$-dimensional manifold. Given $k+1$ 
 smooth vector fields $f_0,\dots, f_{k}$ on $M$, we study control  systems 
 of the form 
\be\label{eq:contrsys-intro}
\dot{q}=f_0(q)+\sum_{i=1}^{k} u_i f_i(q),\qquad q\in M,\quad u\in \ov{B}^{k}_1,
\ee
where $B^{k}_1=\{ u\in\R^{k}\mid \|u\|<1 \}$ is the (open) unit ball contained in $\R^{k}$, and $\ov{B}^{k}_1$ denotes its closure. Systems of the form \eqref{eq:contrsys-intro} are called \emph{control-affine systems}, 
and the geometric aspects of their evolution 
has attracted a lot of interest in the mathematical control community (see e.g. \cite{Agrabook,bullo,jurdj}).

An \emph{admissible trajectory} of \eqref{eq:contrsys-intro} is a Lipschitz continuous curve $q:[0,T]\to M$, $T>0$, for which there exists $u\in L^\infty([0,T], \ov{B}^{k}_1)$ 
such that
\[
	\dot{q}(t)=f_0(q(t))+\sum_{i=1}^{k} u_i(t) f_i(q(t))
\]
holds almost everywhere on $[0,T]$. 

\begin{defi}\label{defi:timeoptcontrproblem}
	The time-optimal control problem associated with \eqref{eq:contrsys-intro} consists into finding the admissible trajectories $q:[0,T]\to M$ of the system that minimize the time needed to join $q(0)$ and $q(T)$, among all the admissible curves. Admissible trajectories that solve the time-optimal control problem associated with \eqref{eq:contrsys-intro} are called \emph{time-optimal trajectories}.
\end{defi}

Candidate time-optimal trajectories are characterized by the Pontryagin maximum principle \cite{PMP} (PMP, in short). Every admissible time-optimal trajectory can be lifted to a Lipschitz continuous trajectory $\lambda:[0,T]\to T^*M$ of 
an associated time-dependent Hamiltonian system (see Section~\ref{sec:PMP} for details). 
Moreover, $\lambda(t)\ne 0$ for every $t\in [0,T]$, and for almost every $t\in [0,T]$ the triple $(q(t),\lambda(t),u(t))$ has the property that 
\begin{equation}\label{eq:maxpmp}
\langle \lambda(t),\sum_{i=1}^{k} u_i(t) f_i(q(t))\rangle=\max_{v\in \ov{B}_1^{k}}\langle \lambda(t),\sum_{i=1}^{k} v_i f_i(q(t)) \rangle.
\end{equation} 
The triple $(q(\cdot),\lambda(\cdot),u(\cdot))$ is said to be an \emph{extremal triple}, and the PMP reduces the study of time-optimal trajectories to the study of extremal triples. We call \emph{extremal trajectory} any admissible trajectory which is part of an extremal triple, so that any time-optimal trajectory is an extremal trajectory, but the converse does not hold in general.
	
	\subsection{Regularity of extremal trajectories}
	Our goal is to establish regularity results for time-optimal trajectories of control-affine systems. Our methods, however, apply to the broader class of extremal ones. 
	
	Given an extremal triple $(q(\cdot),\lambda(\cdot), u(\cdot))$, the control $u$ can be smoothly reconstructed 
	from the maximality condition \eqref{eq:maxpmp}
 whenever $\lambda(t)$ is not simultaneously orthogonal to $f_1(q(t)),\dots, f_{k}(q(t))$.
However, smoothness may stop at times where $\lambda(t)$ annihilates 
$f_1(q(t)),\dots, f_{k}(q(t))$ and, actually,  for any given measurable control $t\mapsto u(t)$, there exist a dynamical system of the form \eqref{eq:contrsys-intro} and an initial datum $q_0\in M$ for which the admissible trajectory driven by $u$ and starting at $q_0$ is time-optimal.
This has been noticed 
in \cite{Sussmann1986} for the single-input case, i.e., when $k=1$, but can be easily extended to the general case.
It makes anyhow sense to investigate regularity 
properties of 
 extremal trajectories 
 for generic systems or, more generally, 
 for systems satisfying low-codimension non-degeneracy conditions. The single-input case, in particular, gave rise to a vast literature (see, e.g., \cite{AgrSig,BonKup,boscain-piccoli,Sch1,Sig1,sussmann1982time,ZelBorCont} and the references therein). 
	
	Recently, the same questions about the regularity of time-optimal trajectories have been posed in the multi-dimensional input case, but only few results are available \cite{agrachev2016switching,CaiDao,CJT03,CJT08,OrieuxRoussaire2019,zelikin2012geometry}. 
	
\begin{defi}\label{defi:O}
	Given an admissible trajectory $q:[0,T]\to M$, we denote by $O_q$  the maximal open subset of $[0,T]$ such that there exists a control $u\in L^\infty([0,T],\ov{B}_1^{k})$, associated with $q(\cdot)$, which is smooth on $O_q$. We also define $\Sigma_q$ (or $\Sigma$, if no ambiguity is possible) as 
	\[
	\Sigma_q=[0,T]\setminus O_q.
	\]
\end{defi}

An isolated point of $\Sigma$ is usually called a \emph{switching time}. 
The accumulation of switching times 
is referred to in the literature as \emph{Fuller phenomenon} (after the pathbreaking work \cite{fuller}), or also \emph{chattering} or \emph{Zeno behavior}.

\begin{defi}[Fuller times]\label{defi:Fuller}
	Let us define $\Sigma_0$ to be the set of isolated points of $\Sigma$. Inductively, we set $\Sigma_j$ to be the set of isolated points of $\Sigma\setminus(\bigcup_{i=0}^{j-1}\Sigma_i)$. A time $t\in \Sigma_j$ is said to be a \emph{Fuller time of order $j$}. Finally, we declare points of 
	\[
	\Sigma_{\infty}=\Sigma\setminus(\bigcup_{j\geq 0}\Sigma_j)
	\]
	to be Fuller times of infinite order.
\end{defi}

\begin{remark}\label{lemma:sigmacount}
For every $j\in \N$, the set $\Sigma_j$ consists of isolated points only, hence it is 
countable. 
\end{remark}

We measure the worst stable behavior of ``generic'' systems of the form \eqref{eq:contrsys-intro} in terms of the maximal order of their Fuller times. The more an instant $t$ is nested among Fuller times of high order, the greater is the number of relations satisfied by the vectors $f_0(q(t)),\dots, f_{k}(q(t))$. 
Transversality theory is then used to guarantee that generically not too many of these conditions can hold at the same point. 
As opposed to the analysis in \cite{boarotto2018}, we restrict ourselves to the case of global frames of everywhere linearly independent vector fields, and the word generic must be intended with respect to this property.

\begin{defi}\label{defi:indepvector}
For every open set $U\subset M$, we denote by
	\begin{itemize}
		\item $\Vect(U)$ the set of smooth vector fields $f$ on $U$, endowed with the $C^\infty$-Whitney topology.
		\item $\Vect(U)^{k+1}$ the set of all $(k+1)$-tuples $\mathbf{f}=(f_0,\dots, f_{k})$ in $\Vect(U)$ with the corresponding product topology.
		\item $\Vect(U)^{k+1}_0$ the set of everywhere linearly independent $(k+1)$-tuples of vector fields on $U$, that is, 
 	\be
		\Vect(U)^{k+1}_0=\left\{ \mathbf{f}
		\in\Vect(U)^{k+1}\big|\, f_0(q)\wedge\dots\wedge f_{k}(q)\ne 0\ \ \text{for every }q\in U \right\}.
	\ee
	We equip $\Vect(U)^{k+1}_0$ with the topology inherited  from $\Vect(U)^{k+1}$.
	\end{itemize}  
\end{defi}

The next statement contains the precise formulation of our main result, which is obtained under the condition $k=2m$, that is, assuming that the number of controlled vector fields is even.

\begin{thm}\label{thm:main}
	\label{t:main}
	Let $m,n\in \N$ be such that $2m+1\le n$. 
	Let $M$ be a $n$-dimensional smooth manifold. There exist a positive integer $K$ {\color{blue} depending only on $n$ 
	}
	and an open and dense set $\mc{U}\subset\Vect(M)^{2m+1}_0$ such that, 
	if the $(2m+1)$-tuple $\mathbf{f}=(f_0,\dots,f_{2m})$ is in $\mc{U}$, 
	then 
	every extremal trajectory $q(\cdot)$
of the time-optimal control problem
	\be
	\dot q=f_0(q)+\sum_{i=1}^{2m}u_if_i(q),\quad q\in M,\ \  u\in\ov{B}_1^{2m},
	\ee
has at most Fuller times of order $K$, i.e.,
	\[
		\Sigma=\Sigma_0\cup \dots \cup \Sigma_{K},
	\]
	where $\Sigma$ and $\Sigma_j$ are as in Definitions~\ref{defi:O} and \ref{defi:Fuller}.	 
\end{thm}
	
	Combining Theorem~\ref{thm:main} and Remark~\ref{lemma:sigmacount}, we deduce that 
	any extremal trajectory $q(\cdot)$
	of a generic control-affine system of the form \eqref{eq:contrsys-intro} with $k=2m$
	is smooth out of a countable set.
	
\subsection{Remarks on the main result and open problems} We conclude this introduction proposing two lines of investigation related to our study. The first one consists into extending our analysis to the case of linearly dependent frames, as the first and the third author have done in \cite[\S 4.1]{boarotto2018} for the single-input case. Even though we expect that similar arguments work also in the multi-input case, the differential structure of the singular locus where the fields $f_0,\dots,f_{2m}$ become dependent is 
more complicated, and needs to be properly investigated.

A different, and possibly more substantial line of research consists into establishing Theorem~\ref{thm:main} for systems of the form \eqref{eq:contrsys-intro} and an \emph{odd number} (greater than one) of controls. The fact that an extremal triple $(q(\cdot),\lambda(\cdot), u(\cdot))$ crosses the singular locus $\{\lambda\in T^*M\mid \langle \lambda, f_i(q) \rangle=0,\, i=1,\dots,2m,\, q=\pi(\lambda)\}$ imposes in the even case a differential condition that we can exploit to begin our iterative arguments (Proposition~\ref{prop:singeven}). 
This condition is based
on the results in \cite{agrachev2016switching}
where the switching behavior in time-optimal  trajectories for multi-input control-affine systems is characterized (see also \cite{CaiDao} for a study in the same spirit for a class of control-affine systems 
issuing from the circular restricted three-body problem). 
In the odd case, 
it is not clear how to derive such a first additional relation at times at which 
an extremal triple $(q(\cdot),\lambda(\cdot), u(\cdot))$ 
crosses the singular locus. 
In the single-input case, this difficulty has been overcome with a suitable analysis of extremal trajectories around Fuller times \cite[Theorem 18]{boarotto2018}, but the arguments there depend decisively on the fact that the control is scalar.
For the general odd case, the problem is 
open, and new ideas are required.

\subsection{Structure of the paper} 
In Section~\ref{sec:preli} we 
present the Pontryagin maximum principle (PMP) to recast the time-optimal problem into its proper geometric framework. Based on the Hamiltonian formalism of the PMP, we establish a differentiation lemma that we will use intensively in the paper (Lemma~\ref{lemma:diff}).
Section~\ref{sec:preli}  also contains 
some general observation on 
the maximal order of the Fuller  times in a set
(Section~\ref{s:Full-o-s}) and 
 classical definitions
about 
jet spaces and transversality theory (Section~\ref{sec:transversality}). 
 Section~\ref{sec:algPreli} collects  additional algebraic material on skew-symmetric matrices that we need in subsequent arguments. 
Sections~\ref{sec:acc} and \ref{s:acc-notinv} are devoted to the recursive  characterization of dependence conditions holding at 
accumulations of Fuller times, when the Goh matrix is, respectively, invertible and singular.
Finally, in Section~\ref{sec:tran}, we
conclude 
the proof of the main result, Theorem~\ref{thm:main}.

\section{Main technical tools}\label{sec:preli}

\subsection{The Pontryagin maximum principle
}\label{sec:PMP}
	
	Let us introduce some technical notations that we will employ extensively throughout the rest of the paper. Let $\pi:T^*M\to M$ be the cotangent bundle, and $s\in \Lambda^1(T^*M)$ be the tautological Liouville one-form on $T^*M$. The non-degenerate skew-symmetric form $\sigma=ds\in \Lambda^2(T^*M)$
endows $T^*M$ with a canonical symplectic structure.
	
	With any $C^1$ function $p:T^*M\to \R$ let us associate its Hamiltonian lift $\vec{p}\in
	C(T^*M,TT^*M)$ by the condition
	\be\label{eq:Hamlift}
		\sigma_\lambda(\cdot,\vec{p})=d_\lambda p.
	\ee
	
	Fix $\mathbf{f}=(f_0,\dots, f_{2m})\in \Vect(M)^{2m+1}$. 
	The Pontryagin Maximum Principle (PMP, for short) \cite{PMP} gives then a 
	necessary condition satisfied by candidate time-optimal trajectories of
\be\label{eq:contrsys}
\dot{q}=f_0(q)+\sum_{i=1}^{2m} u_i f_i(q),\qquad q\in M,\quad u\in \ov{B}^{2m}_1,
\ee
 recalled in the theorem below. Introducing the control-dependent Hamiltonian function $\mc{H}: T^*M \times \R^{2m}\to\R$ by
	\be\label{eq:maxH}
		\mc{H}(\lambda,v)=\langle \lambda, f_0(q)+\sum_{i=1}^{2m}v_if_i(q) \rangle, \ \ q=\pi(\lambda),
	\ee
	the precise statement is the following.
	
	\begin{thm}[PMP]\label{thm:PMP}
		Let $q:[0,T]\to M$ be a time-optimal trajectory of \eqref{eq:contrsys}, associated with a 
		control $u(\cdot)$.
		Then there exists an absolutely continuous curve
		$\lambda:[0,T]\to T^*M$ such that
		$(q(\cdot),\lambda(\cdot),u(\cdot))$ is an extremal triple, i.e., in terms of the control-dependent Hamiltonian $\mathcal{H}$ introduced in \eqref{eq:maxH}, one has 
		\begin{align}
			&\lambda(t)\in T^*_{q(t)}M\setminus \{0\},\quad \forall t\in [0,T],\\
			\label{eq:maxPMP}
			&		\mathcal{H}(\lambda(t),u(t))=\max\{\mathcal{H}(\lambda(t),v)\mid v\in\ov{B}_1^{2m}\}\ \ \mathrm{for\ a.e.\ }t\in[0,T],\\
			&\dot{\lambda}(t)=\vec{\mathcal{H}}
			(\lambda(t),u(t)),\quad \mathrm{for\ a.e.\ }t\in[0,T].\label{eq:dynPMP}
		\end{align}
	\end{thm}
	
	\begin{defi}
		For any extremal triple $(q(\cdot),\lambda(\cdot),u(\cdot))$, we call the corresponding trajectory $t\mapsto q(t)$ a time-extremal trajectory, and the curve $t\mapsto \lambda(t)$ its associated time-extremal lift.
	\end{defi}
	
	For every $i=0,\dots,2m$, let us define the smooth functions $h_i:T^*M\to\R$ by
	\[
		h_i(\lambda):=\langle \lambda, f_i(q)\rangle,\ \ q=\pi(\lambda).
	\]
	
More generally, let $k$ be an integer and $D=i_1\cdots i_k$ a multi-index  of 
$\{0,1,\dots,2m\}$, and let $|D|:=k$ be the length of $D$.
A multi-index $D=i\cdots ij$ with $k$ consecutive occurrences of the index $i$ is denoted as $D=i^kj$. 
We use $f_D$ to denote the vector field defined by
\be f_D=\left[f_{i_1},\left[\cdots,\left[f_{i_{k-1}},f_{i_k}\right]\cdots\right]\right],\ee
and $h_D$ to denote the smooth function on 
$T^*M$ given by $\langle\lambda,f_D\rangle$ for $\lambda\in T^*
M$.

By a slight abuse of notations, given a time-extremal triple $(q(\cdot),\lambda(\cdot),u(\cdot))$ defined on $[0,T]$, we define 
	$h_i(t):=h_i(\lambda(t))$ for every $i=1,\dots,{2m}$ and $t\in [0,T]$. Throughout the rest of the paper, we further extend this convention in the following way: whenever $\varphi:T^*M\to \R$ is a scalar function defined on 
	$
	T^*M$ and $t\mapsto \lambda(t)$ is an integral curve of $\vec{\mc{H}}$, we denote by $\varphi(t)$ the evaluation of $\varphi$ at $\lambda(t)$ if no ambiguity is possible.

	Denote by $I$ the set $\{1,\dots,2m\}$ and by $h_I$ 
	the map $h_I: T^*M\to\R^{2m}$ defined by
	\be\label{eq:h}
		h_I(\lambda)=(h_1(\lambda),\dots,h_{2m}(\lambda)).
	\ee

	Let us first recall that the time-extremal control $u$ is smooth (up to modification on a set of measure zero) on the open set $R_q:=\{t\in [0,T]\mid h_I(t)\neq 0\}$, i.e., in terms of the set $\Sigma_q$ introduced in Definition~\ref{defi:O}, 
	\begin{equation}\label{eq:Sigmainh=0}
	\Sigma_q\subset \{t\in [0,T]\mid h_I(t)=0\}. 
	\end{equation}  
	Indeed, the maximality condition \eqref{eq:maxPMP} provided by the PMP yields  
	the explicit characterization 
	\be\label{eq:upmp}
		u(t)=\frac{h_I(t)}{\|h_I(t)\|},\quad t\in R_q.
	\ee
	Therefore an extremal trajectory on $R_q$ is an integral curve of the vector field 
	\[
		\lambda\mapsto \vec{\mc{H}}\left(\lambda,\frac{h_I(\lambda)}{\|h_I(\lambda)\|}\right),
	\]
	which is well-defined and smooth on $T^*M\setminus \{ \lambda\in T^*M\mid h_I(\lambda)=0 \}$. 
	In particular, its integral curves are smooth as well. 
	
	We also recall the following differentiation formula along a time-extremal lift $t\mapsto \lambda(t)$, which follows as a consequence of the symplectic structure on $T^*M$ (see \cite[Section 3.3]{AbrahamMarsden}).
	
	\begin{prop}\label{prop:diffextremal}
		Let $\varphi:T^*M\to\R$ be a $C^1$ function, and let $\lambda:[0,T]\to T^*M$ be 
		a solution of
		\eqref{eq:dynPMP}
		corresponding to a control $u:[0,T]\to \ov{B}^{2m}_1$. Then
		\be\label{eq:fordiff}
			\frac{d}{dt}\varphi(\lambda(t))=\{h_0,\varphi \}(\lambda(t))+\sum_{i=1}^{2m} u_i(t)\{ h_i,\varphi \}(\lambda(t))\ \ \text{a.e. on } [0,T].
		\ee
	\end{prop} 
	In particular, Proposition~\ref{prop:diffextremal} implies that 
for every	$X\in\Vc(M)$ and every extremal triple associated with \eqref{eq:contrsys}
	the identity
	\be\label{eq:diffvectfield}
		\frac{d}{dt}\langle \lambda(t),X(q(t)) \rangle=\langle \lambda(t),[ f_0+\sum_{i=1}^{2m} u_i(t)f_i,X ](q(t))\rangle
	\ee
	holds true for a.e. $t$ (here we apply the proposition to $\varphi(\lambda)=\left\langle \lambda,X(\pi(\lambda))\right\rangle$
	). 

Denote by  $M_{j,k}(\R)$ the set of $j\times k$ matrices with real entries and let $M_j(\R)=M_{j,j}(\R)$. 
	We introduce the map
	\begin{align}\label{eq:Goh}
		H_{II}:T^*M&\to M_{2m}(\R),\\
			\lambda&\mapsto 
			(\{ h_i,h_j \}(\lambda))_{i,j=1}^{2m}.
	\end{align}
	For every $\lambda\in T^*M$, the skew symmetric matrix $H_{II}(\lambda)$ is called the \emph{Goh matrix}.
	Defining $h_{0I}:T^*M\to M_{2m,1}(\R)$ to be the vector-valued function $(h_{0i}(\lambda))_{i=1}^{2m}$ and differentiating $h_I$ along a time-extremal triple, we find by the previous considerations 
	that 
	\be\label{eq:firstdiff}
		\dot{h}_I(t)=h_{0I}(t)-H_{II}(t)u(t)
	\ee
	for a.e. $t$ 
	(notice that the minus sign is a consequence of considering the transposition in \eqref{eq:fordiff}). In particular, within the set $R$, the dynamics of $h_I$ are described by 
	\[
		\dot{h}_I(t)=h_{0I}(t)-H_{II}(t)\frac{h_I(t)}{\|h_I(t)\|}.
	\]

\subsection{A differentiation lemma}

We present in this section a 
result that we will extensively use in the paper. It concerns the 
differentiation along an extremal curve of a smooth  
function on $T^*M$
that vanishes at a converging sequence of times.

\begin{lemma}\label{lemma:diff}
Let  $(q(\cdot),\lambda(\cdot),u(\cdot))$ be an extremal triple on $[0,T]$ associated with \eqref{eq:contrsys}. 
	Assume that there exists a 
	sequence of times $(t_l)_{l\in\N}$ in $[0,T]$ such that 
$t_l\to t^*\in [0,T]$ and $t_l\ne t^*$ for every $l\in \N$.
	Then  there exists $u^*\in \ov{B}^{2m}_1$ such that, for every 	smooth function $\varphi:T^*M\to \R$ satisfying $\varphi(\lambda(t_l))=0$ for every $l\in\N$,
	\be\label{eq:diff}
	\{h_0,\varphi\}(\lambda(t^*))+\sum_{i=1}^{2m}u_i^*\{ h_i,\varphi \}(\lambda(t^*))=0.
	\ee
\end{lemma}

\begin{proof}
	Since $u(\cdot)\in  L^\infty([0,T],\ov{B}_1^{2m})$, there exists  a subsequence $(t_{l_w})_{w\in\N}$ such that the limit
	\[
		u^*:=\lim_{w\to\infty}\frac{1}{t^*-t_{l_w}}\int_{t_{l_w}}^{t^*}u(t)dt
	\]
	exists and belongs to $\ov{B}^1_{2m}$. 

Consider a smooth function $\varphi:T^*M\to \R$ such that $\varphi(\lambda(t_l))=0$ for every $l\in\N$.
	By continuity we have $\varphi(\lambda(t^*))=0$, so that by Proposition~\ref{prop:diffextremal} for every $l\in\N$ we can write
	\begin{align}\label{eq:howtodiff}
		0&=\frac{\varphi(\lambda(t^*))-\varphi(\lambda(t_l))}{t^*-t_l}=\frac{1}{t^*-t_l}\int_{t_l}^{t^*}\frac{d}{dt}\varphi(\lambda(t))dt\\&=\frac{1}{t^*-t_l}\int_{t_l}^{t^*}\big(\{ h_0,\varphi \}(\lambda(t))+\sum_{i=1}^{2m} u_i(t)\{ h_i,\varphi \}(\lambda(t))\big) dt.		
	\end{align}
	Rewriting \eqref{eq:howtodiff} along the subsequence $t_{l_w}$ and taking the limit as $w\to\infty$ permits then to conclude, since 
	 $t\mapsto \{ h_i,\varphi \}(\lambda(t))$ is absolutely continuous for every $i=0,\dots,2m$.
\end{proof}

\subsection{Fuller order of a set}\label{s:Full-o-s}

For a subset $\Xi$ of $\R$ we denote by $\Xi_0$ its subset made of isolated points and, 
 inductively, by $\Xi_j$ the set of isolated points of 
 $\Xi\setminus(\bigcup_{i=0}^{j-1}\Xi_i)$, $j\ge1$.

\begin{defi}
We say that $\Xi$ has \emph{Fuller order $k\in \N$} 	
if $\Xi=\Xi_0\cup\cdots\cup \Xi_k$ and $\Xi_k\ne \emptyset$. 
We say that 
$\emptyset$ has \emph{Fuller order $-1$} and that
$\Xi$ has \emph{Fuller order $\infty$} if $\Xi\setminus(\bigcup_{i=0}^{k}\Xi_i)
\ne \emptyset$ for every $k\in \N$.
\end{defi}

\begin{remark}
The notion of Fuller order is strictly related to the one of Cantor-Bendixson rank: if $X$ is a topological space (in particular, a subset of $\R$ with the induced topology)
the \emph{Cantor-Bendixson rank of $X$} is the least 
ordinal such that $X^{(\alpha)}=X^{(\alpha+1)}$, where $X^{(1)}=\{x\in X\mid x\in \overline{X\setminus\{x\}}\}$ is the \emph{derived subset of $X$}, $X^{(\alpha+1)}=(X^{(\alpha)})^{(1)}$, and $X^{(\beta)}=\cap_{\alpha<\beta}X^{(\alpha)}$.  
For \emph{scattered} sets, i.e., sets such that $X^{(k)}=\emptyset$ for some $k\in\N$,  the Cantor-Bendixson rank is equal to the Fuller order plus 1. 
For perfect sets, on the contrary, the Fuller order is infinite and the Cantor-Bendixson rank  is zero. 

The properties of the Fuller order described in the following two results have been probably already observed in the context of Cantor-Bendixson rank 
but we were not able to find a precise reference for them.
\end{remark}

\begin{lemma}\label{lem:Fuller-order}
Let 
$\Xi,\mathfrak{S}$ be two subsets of $\R$. 
If $\Xi$ has Fuller order at least $k$ and $\mathfrak{S}$ has Fuller order at most $j$, with $k>j\ge 0$, then 
$\Xi\setminus \mathfrak{S}$ has Fuller order at least $k-j-1$.
\end{lemma}
\begin{proof}
Without loss of generality 
$\Xi$ has order $k$ and $\mathfrak{S}$ has order $j$. 
Notice that it is enough to prove the lemma in the case $j=0$, since every set $\mathfrak{S}_i$, $i=0,\dots,h$, is of Fuller order $0$ and 
\[\Xi\setminus \mathfrak{S}=(\cdots((\Xi\setminus \mathfrak{S}_0)\setminus \mathfrak{S}_1)\cdots )\setminus \mathfrak{S}_j).\]

Let us prove the property by induction on $k$, assuming that $\mathfrak{S}=\mathfrak{S}_0$. 
In the case $k=1$, we just need to notice that $\Xi\setminus \mathfrak{S}$ is nonempty and hence has nonnegative Fuller order.
Assume now that the property holds for $k-1$ and let us prove it for $k$. 
Consider a point $x\in \Xi_k$. If $x$ is in $\mathfrak{S}$, then there exists a neighborhood of $x$ which does not contain any  point of $\mathfrak{S}$ except $x$. 
Since $x$ is a density point for $\Xi_{k-1}$, we deduce that there exist points in $\Xi_{k-1}$ at positive distance from $\mathfrak{S}$. Hence $\Xi\setminus\mathfrak{S}$ has Fuller order at least $k-1$. Assume now that $x$ is in $\Xi\setminus\mathfrak{S}$. 
Notice that, by the induction hypothesis, for every neighborhood $U$ of $x$,
the set $U\cap ((\Xi_0\cup\cdots \cup \Xi_{k-1})\setminus \mathfrak{S})$ has Fuller order at least $k-2$. We can then extract a sequence in $((\Xi_0\cup\cdots \cup \Xi_{k-1})\setminus \mathfrak{S})_{k-2}$ converging to $x$. We deduce that 
$\Xi\setminus\mathfrak{S}$ has Fuller order at least $k-1$.
\end{proof}

As an immediate consequence, we get the following result. 

\begin{corollary}
\label{cor:imbriques}
Let $k\ge 1$, $j\ge 0$, and $\Xi\subset \R$ be the union of $\Xi^1,\dots,\Xi^k$. 
If $\Xi^i$ has Fuller order at most $j$ for every $i\in\{1,\dots,k\}$, then $\Xi$ has Fuller order at most $k(j+1)$.  
\end{corollary}

\subsection{Jet spaces and transversality}\label{sec:transversality} 

Following \cite{CJT06}, for any nonempty open subset  $U$ of $M$ 
	we introduce:
	\begin{itemize}
		\item $JTU$: the jet space of the smooth vector fields on $U$,
		\item  $J^NTU$, $N\in\N$: the jet space of order $N$,
		\item  $J_{2m+1}^NTU$: the fiber product $J^NTU\times_U\dots\times_U J^NTU$ of $2m+1$ copies of $J^NTU$,
		\item $J_{q}^NTU$: the fiber of $J^NTU$ at $q\in U$,
		\item $J_{2m+1,q}^NTU$: the fiber of $J_{2m+1}^NTU$ at $q\in U$, 
		\item $T_{2m+1,N}$ the typical fiber of $J_{2m+1}^NTU$.
	\end{itemize}
	The spaces $JTU$, $J^NTU$ and $J_{2m+1}^NTU$ are endowed with the Whitney $C^\infty$ topology.
	
If $N$ is a positive integer and $f\in \Vect(U)$ (respectively, $\mathbf{f}\in \Vect(U)^{2m+1}$),
 we use $j^N(f)$ and $j_{q}^N(f)$ (respectively, $j^N(\mathbf{f})$ and $j_{q}^N(\mathbf{f})$) to denote respectively the jet of order $N$ associated with $f$ (respectively, the $(2m+1)$-tuple of jets of order $N$ associated with $\mathbf{f}$) and its evaluation at $q\in U$ (respectively, the evaluation of $j^N(\mathbf{f})$ at $q\in U$).

	Fix $N\in\N$ and let $P(n,N)$ be the set of all polynomial mappings 
	\be
		G:=\left(G^1,\dots, G^n\right):\R^n\to\R^n,\ \ \mathrm{deg}(G^i)\le N,\ \ \text{for every }\, 1\le i\le n.
	\ee
	Similarly, we call $P(n,N)^{2m+1}$ the set of all $(2m+1)$-tuples of elements in $P(n,N)$, that is,
	\be\label{eq:mtuples}
		P(n,N)^{2m+1}=\left\{\left(Q_1,\dots, Q_{2m+1}\right)\mid Q_i\in P(n,N), 1\le i\le 2m+1  \right\}.
	\ee
Assume from now on that $U$ is the domain of a coordinate chart $(x,U)$ centered at some $q\in U$. This allows one to identify  
the typical fiber $T_{2m+1,N}$ of 
	$J^N_{2m+1}TU$ with $P(n,N)^{2m+1}
	$ as explained below.
	There is a standard way \cite{BonKup} of introducing coordinates on the semi-algebraic set
	\be\label{eq:Omega}
		\Omega:=\left\{ 
		\left(Q_1,\dots, Q_{2m+1}\right)\in P(n,N)^{2m+1}\mid Q_1(0)\wedge \dots\wedge Q_{2m+1}(0)\ne 0 \right\}\subset P(n,N)^{2m+1},
	\ee
	which we briefly recall.	
		
		Let $\mc{K}_0=\{0\}$, and $\mc{K}_k$ be the set of $k$-tuples of ordered integers in $\{1,\dots,n\}$. If $f:\R^n\to \R$ is a homogeneous polynomial of degree $k$, and $\xi=(\xi_1,\dots,\xi_k)\in (\R^{n})^k$, the polarization of $f$ along $\xi$ is the real number
		\be
			Pf(\xi):=D_{\xi_1}\dots D_{\xi_k}f,
		\ee
where, for every $\eta\in \R^n$, $D_\eta f$ denotes the 
directional derivative
of $f$ along $\eta$.

Given $\wh{Q}\in \Omega$, we complete $\left(\wh{Q}_1(0),\dots, \wh{Q}_{2m+1}(0)\right)$ to a basis of $\R^n$ with $n-2m-1$ vectors $v_{2m+2},\dots, v_n\in \R^n$. There exists a neighborhood $V\subset \Omega$ of $\wh{Q}$ such that the map 
\begin{align}
	\mathrm{ev}:V&\to (\R^n)^n\\
			    Q&\mapsto \left( Q_1(0),\dots Q_{2m+1}(0),v_{2m+2},\dots, v_n \right)
\end{align}
associates with any element $Q\in V$ a basis of $\R^n$.  For $1\le i \le n$ and $Q\in V$, we also employ the notation $\mathrm{ev}(Q)_i$ to refer to the $i$-th component of $\mathrm{ev}(Q)$. In particular $\mathrm{ev}(Q)_i\in \R^n$. This allows to introduce a coordinate chart $X_V$ on $V$, in such a way that every $Q=\left(Q_1,\dots, Q_{2m+1}\right)\in V$ can be written with coordinates
	 \be
	 	 \left\{ X_{i,\sigma}^j\,\bigg\vert\, 1\le i\le 2m+1,\, 1\le j\le n,\, \sigma\in\mc{K}_k,\, 0\le k\le N  \right\},
	\ee
	 where the element $X_{i,\sigma}^j$ denotes the polarization of the $j$-th coordinate of the homogeneous part of degree $k=|\sigma|$ of $Q_i$ along the element 
	 $\left( \mathrm{ev}(Q)_{\sigma_1},\dots ,\mathrm{ev}(Q)_{\sigma_k} \right)$.

	Consider the now the chart  $(X_V,x)$ on the domain $V\times U\subset \Omega\times M$.
If $\sigma\in\mc{K}_k$, define $\sigma!=\sigma_1!\dots\sigma_k!$ and $x^\sigma=x_1^{\sigma_1}\dots x_k^{\sigma_k}$. In local coordinates, $Q_i$ is represented by
\be
	Q_i=\frac{\partial}{\partial x_i}+\sum_{\substack{1\le k\le N \\ \sigma\in \mc{K}_k}}\frac{x^\sigma}{\sigma!}X_{i,\sigma},\ \ X_{i,\sigma}=\sum_{j=1}^n X_{i,\sigma}^j\frac{\partial}{\partial x_j},
\ee
and 
$X_{i,\sigma}$ is a constant vector field.

If $1\le i\ne k\le 2m+1$, in these local coordinates we see that $[Q_i,Q_k](0)=Q_{ik}(0)=X_{k,i}-X_{i,k}$ and similarly, if $Q_{i^lk}$ denotes the $l$-fold iterated bracket $\mathrm{ad}_{Q_i}^l(Q_k)$, we deduce inductively that $Q_{i^lk}(0)=X_{k,i^l}+R_{i,k,l}$, where $R_{i,k,l}$ is a polynomial in the coordinates $X_{s,\sigma}^a$, with $1\le a\le n$, $1\le s\le 2m+1$, $|\sigma|\le l$ and $\sigma\ne j^l$. Similar computations can be carried out for all iterated brackets.
	
\begin{remark}\label{rem:coor}
Let $\left( (x,\psi),\pi^{-1}(U) \right)$ be the induced chart on $T^*U$, where $\psi=(\psi_r)_{r=1}^{n}$. In particular, we use $\lambda_\psi$ to denote the elements of $T_0^*M$ given in coordinates by $(0,\psi)$. The typical fiber $\widehat{T}_{2m+1,N}$ of the vector bundle $J^N_{2m+1}TU\times_U T^*U$ is isomorphic to $P(n,N)^{2m+1}\times \R^n$.  Clearly, $h_{ik}(\lambda_\psi)=\langle \psi,X_{k,i}\rangle-\langle \psi,X_{i,k}\rangle$ and, for $l\geq 1$, 
\be
h_{i^lk}(\lambda_\psi)=\langle \psi,Q_{i^lk}(0)\rangle=\langle \psi,X_{k,i^l}\rangle+\langle \psi,R'_{i,k,l}\rangle,
\ee
where $R'_{i,k,l}$ is a polynomial in the coordinates $\psi_r,X_{s,\sigma}^a$ with $1\le a,r\le n$, $1\le s\le 2m+1$, $|\sigma|\le l$ and $\sigma\ne j^l$. By an induction argument, 
$h_D(\lambda_\psi)$, with $D$ a multi-index, can be expressed as a polynomial function in terms of the coordinates $\psi_r,X_{s,\sigma}^a$. Therefore, this choice of the
chart $(X_V,x)$ allows one to see every $h_D$ and $h_D\circ \mathrm{ev}$  as a real-valued function on $J^N_{2m+1}TU\times_UT^*U$ and on its typical fiber $\widehat{T}_{2m+1,N}$, respectively, where $N$ is large enough. This will also be the case for any polynomial function in the $h_D$'s. 
\end{remark}	

The following result follows by standard transversality arguments (see, e.g., \cite{Abraham,SMT}). 

\begin{lemma}[Transversality Lemma]\label{lem:TL}
	Let $N\in\N$. 
	Let $\mc{B}$ be a closed subset of $J^N_{2m+1}TM$ and
	assume that 
for every $q\in M$ there exists a coordinate chart $(x,U)$ centered at $q$ such that 	
	$\mc{B}\cap J^N_{2m+1}TU$ is semi-algebraic in the coordinates $(X_V,x)$ introduced above.  
For every $q\in M$ let $\mc{B}_{q}:=\mc{B}\cap J_{2m+1,q}^NTM$. 
Let $\mc{V}$ be the open subset of $\Vect(M)_0^{2m+1}$ made of the $(2m+1)$-tuples $\mathbf{f}=(f_0,\dots,f_{2m})$ such that, for every $q\in M$, 
$j^{N}_{q}(\mathbf{f})\not\in \mc{B}_{q}$. Assume that $\mc{B}_{q}$ has codimension larger than or equal to $n+1$ in $J_{2m+1,q}^NTM$ for every $q\in M$. 
Then 
$\mc{V}$ is also dense in $\Vect(M)_0^{2m+1}$.
\end{lemma}

\section{Algebraic considerations} \label{sec:algPreli}
	\subsection{Decomposition of skew-symmetric matrices} We collect in this section some general facts regarding the algebraic structure of skew-symmetric matrices.
	For any $l\in\N$, we recall that the notation $\mathfrak{so}(l)$ stands for the linear space of $l\times l$ skew-symmetric real matrices. We begin by {\color{blue} recalling some useful properties concerning the Pfaffian of a skew-symmetric matrix}.
	
	\begin{lemma}\label{lemma:Pf}
		Let $A\in\mathfrak{so}(2m)$. 
		Then the following properties hold true.
		\begin{itemize}
			\item [i)] $\det(A)=\Pf(A)^2$, where
			 $\Pf(A)$, called the \emph{Pfaffian of $A$}, is a homogeneous polynomial
			 in the entries of $A$ of degree $m$.
			\item [ii)] There exists a $2m\times 2m$ skew-symmetric matrix $\adjp(A)$, called the \emph{adjoint Pfaffian of $A$}, such that
its entries are homogeneous polynomial of degree $m-1$ in the entries of $A$ and
			\[
				\adjp(A)A=\Pf(A)\mathrm{Id}_{2m}.
			\]
		\end{itemize} 
	\end{lemma}
	
	\begin{proof}
		Item i) is classical, and we refer the reader to \cite{Led} for a proof.
{\color{blue} Concerning Item ii), it can be found, for instance, in \cite[Equation (3.2)]{pfaffian}.}		
%
	\end{proof}
	
	The next proposition collects a list of useful properties valid for general skew-symmetric matrices of size $k$.
	
	\begin{prop}\label{prop:princmin}
		Let $k\in \N$ and $A\in\mathfrak{so}(k)$ be nonzero. 
		Then the following holds true.
		\begin{itemize}
			\item [i)] The rank of $A$ is an even integer $1\leq 2m_0\leq k$ and there exists a nonzero principal minor of order $2m_0$. As a consequence, there exists a permutation matrix $P$ such that 
			\be\label{eq:UM}
				P^TAP=\left(\begin{array}{cc}
					A_1 & A_2 \\
					-A_2^T & A_3
				\end{array}\right), 
			\ee
			where  $A_1\in \mathfrak{so}(2m_0)$ is invertible, 
			$A_2\in M_{2m_0,k-2m_0}(\R)$, 
			and $A_3\in \mathfrak{so}(k-2m_0)$.
			
			\item[ii)] With $P^TAP$ presented as in \eqref{eq:UM} one has
			\[
				\ker(P^TAP)=\mathrm{span}\{ (-A_1^{-1}A_2x_2,x_2)\mid x_2\in \R^{k-2m_0} \}.
			\]
In particular, 
$A_1$, $A_2$, and $A_3$ satisfy the relation
			\be\label{eq:parskew}
				A_2^TA_1^{-1}A_2+A_3=0.
			\ee
			\item[iii)] Let $e_1,\dots,e_{k-2m_0}$
			be the canonical basis of $\R^{k-2m_0}$. Define
			\[
				v_i=\left( -\adjp(A_1)A_2e_i,\Pf(A_1)e_i\right),\ \ 
				1\leq i\leq k-2m_0,
			\]
			where $\adjp(A_1)$ denotes the adjoint Pfaffian of $A_1$ introduced in Lemma~\ref{lemma:Pf}.
			Then the family $v_1,\dots,v_{k-2m_0}$
			is a basis of $\ker (P^TAP)$, and the coordinates of each $v_i$, for $i=1,\dots,k-2m_0$, are homogeneous polynomials 
of degree $m_0$ in the entries of $A$. 
		\end{itemize}
	\end{prop}

\begin{proof}
	We begin by i). First note that the conclusion is equivalent to prove that $A$ admits a $2m_0\times 2m_0$ nonzero principal minor, i.e., the determinant of an $2m_0\times 2m_0$ principal submatrix. Recall that, for $1\leq l\leq k$, the coefficient of $(-1)^lx^{k-l}$ of the characteristic polynomial of any $k\times k$ matrix is equal to the sum of its $l\times l$ principal minors. If $A$ is a $k\times k$ skew-symmetric matrix, notice that its principal submatrices are themselves skew-symmetric. One deduces that the coefficients of $(-1)^lx^{k-l}$ in the characteristic polynomial $P_A$ of $A$ are zero if $l$ is odd and sums of squares if $l$ is even, according to i) of Lemma~\ref{lemma:Pf}. Moreover, if the rank of $A$ is equal to $2m_0$, then 
	$P_A(x)=x^{k-2m_0}Q(x)$ with $Q(0)\neq 0$ since $A$ is diagonalizable over $\mathbb{C}$. Hence 
	the coefficient of $x^{k-2m_0}$ of $P_A$ is nonzero, yielding the existence of a $2m_0\times 2m_0$ nonzero principal minor.

		We pass now to~Point ii). Let us consider any element $w=(w_1,w_2)\in \ker(P^TAP)$. Computing the product $P^TAP w=0$, and recalling that $A_1$ is invertible, we obtain the relations
		\[
			w_1=-A_1^{-1}A_2w_2,\quad (A_2^TA_1^{-1}A_2+A_3)w_2=0.
		\]
		By assumption, $\ker(P^TAP)$ has dimension $k-2m_0$, therefore there exists a basis $w_2^1,\dots,w_2^{k-2m_0}$ of $\R^{k-2m_0}$, such that the elements
		\[
			(A_1^{-1}A_2w_2^i,w_2^i),\quad i= 1,\dots,k-2m_0 ,
		\]
		belong to $\ker(P^TAP)$ and are linearly independent. In particular the $(k-2m_0)\times (k-2m_0)$ skew-symmetric matrix $(A_2^TA_1^{-1}A_2+A_3)$ has a $(k-2m_0)$-dimensional kernel, and therefore it is the zero matrix.
		
		As for Point iii), it is sufficient to notice that 
		the elements
		\[
			v_i:=\Pf(A_1)(-A_1^{-1}A_2e_i,e_i),\quad i=1,\dots,k-2m_0,
		\]
		form a basis of $\ker(P^TAP)$ and that, by Lemma~\ref{lemma:Pf}, 
		\[v_i =(-\adjp(A_1)A_2e_i,\Pf(A_1)e_i),\quad i=1,\dots,k-2m_0,\]
		 and, in particular, the coordinates of $v_i$ are homogeneous polynomials of degree $m_0$ in the entries of $A$. 
	\end{proof}

	\subsection{Consequences on the structure of the Goh matrix}
	
	We apply here below Proposition~\ref{prop:princmin} to the skew-symmetric Goh matrix $H_{II}$ defined in \eqref{eq:Goh}. 
	
	Let $(q(\cdot),\lambda(\cdot),u(\cdot))$ be a time-extremal triple of \eqref{eq:contrsys}, and assume that $t^*\in [0,T]$ is such that $1\leq \Rank{(H_{II}(t^*))}=2m_0\leq 2m$. Then, up to a permutation of the basis of $\R^{2m}$ we can present $H_{II}(t^*)$ in the block form
	\be\label{eq:blockform}
		H_{II}(t^*)=\left(\begin{array}{cc}
			H_{II}^{2m_0}(t^*) & E(t^*) \\
			-E(t^*)^T & F(t^*)
			\end{array}\right),
	\ee
	where $H_{II}^{2m_0}(t^*)\in M_{2m_0}(\R)$ and $F(t^*)\in M_{2(m-m_0)}(\R)$ are skew-symmetric matrices, $H_{II}^{2m_0}(t^*)$ is invertible and $E(t^*)\in M_{2m_0,2(m-m_0)}(\R)$. 
	Then the following holds true.
	
	\begin{prop}\label{prop:paramker}
		 There exist a relatively open interval $\mc{I}\subset [0,T]$ containing $t^*$, and smooth functions $v_1,\dots,v_{2(m-m_0)}:[0,T]\to \R^{2m}$ such that:
		 \begin{itemize}
		 	\item [i)] for every $i=1,\dots,2(m-m_0)$ and every $t\in \mc{I}$, letting $e_i$ be the $i$-th element of the canonical basis of $\R^{2(m-m_0)}$, 
		 	\[
		 		v_i(t)=\left(  \begin{array}{c} -\adjp( H_{II}^{2m_0}(t))E(t) e_i \\ \Pf(H_{II}^{2m_0}(t))e_i  \end{array}  \right)
		 	\] 
		 	is a $2m$-dimensional vector whose components are homogeneous polynomials of degree $m_0$ in the entries $h_{ij}(t)$ of the Goh matrix;
			\item [ii)] if $t\in \mc{I}$ is such that $\Rank{(H_{II}(t))}=2m_0$, then 
			\[
				\ker(H_{II}(t))=\Span{\{v_1(t),\dots,v_{2(m-m_0)}(t)\}};
			\]
			\item [iii)] if $t\in \mc{I}$ is such that $\Rank{(H_{II}(t))}=2m_0$, the non-trivial relations expressed by the matrix equality
			\[
				E(t)^T\adjp (H_{II}^{2m_0}(t))E(t)+\Pf(H_{II}^{2m_0}(t))F(t)=0
			\]
			are homogeneous polynomial relations of degree $m_0+1$ in the entries $h_{ij}(t)$ of the Goh matrix.  
		\end{itemize}
	\end{prop}

	\section{
Iterated accumulations of points in $\Sigma$ with invertible Goh matrix	
}
	\label{sec:acc}\label{s:acc-2m}

	Let $(q(\cdot),\lambda(\cdot),u(\cdot))$ be an extremal triple of \eqref{eq:contrsys}. Consider the set
	\be\label{eq:sigmaet}
		\Sigma^{2m}:=\Sigma\cap \{ t\in [0,T]\mid \det{H_{II}(t)}\ne 0 \},
	\ee
	where $\Sigma$ is  the set constructed in Definition~\ref{defi:O}.  
	In analogy with Definition~\ref{defi:Fuller}, we define $\Sigma^{2m}_0$ to be the set of isolated points of $\Sigma^{2m}$ and, inductively, we set $\Sigma^{2m}_j$ to be the set of isolated points of $\Sigma^	{2m}\setminus(\bigcup_{i=0}^{j-1}\Sigma^{2m}_i)$.

	The starting point of the study of accumulations of singularities in $\Sigma^{2m}$ is the  following result.
	
	\begin{prop}\label{prop:singeven}
		Let $t^*\in \Sigma^{2m}\setminus \Sigma^{2m}_0$. Then 
		\be\label{eq:firstres}
			\|H_{II}(t^*)^{-1}h_{0I}(t^*)\|=1.
		\ee
	\end{prop}
	
	\begin{proof}
		Since $t^*\in\Sigma^{2m}\subset \Sigma$, 
		{\color{blue} we have that
		$\det(H_{II}(t^*))\neq 0$ and
		we deduce from \eqref{eq:Sigmainh=0} that 
		$h_I(t^*)=0$}. Moreover, since $t^*\notin\Sigma_0^{2m}$, 
		there exists a nontrivial sequence $(t_l)_{l\in\N}\subset \Sigma^	{2m}$ converging to $t^*$ such that $h_I(t_l)=0$ for every $l\in\N$. 		
		Applying Lemma~\ref{lemma:diff} to $\varphi=h_i$, $i\in I$, 
		we 
		infer the existence of $u^*\in \ov{B}_1^{2m}$ such that 
		\[
			h_{0I}(t^*)-H_{II}(t^*)u^*=0,
		\]
		that is, we deduce that $h_{0I}(t^*)\in H_{II}(t^*)\ov{B}_1^{2m}$. 
		
		Assume by contradiction that 
		{\color{blue} 
		$h_{0I}(t^*)\in H_{II}(t^*)B_1^{2m}$.
		Then we deduce from \cite[Theorem 3.4]{agrachev2016switching}
		that $h_I$ vanishes identically in a relative neighborhood $\mc{I}\subset [0,T]$ of $t^*$. 
Note that 	\cite[Theorem 3.4]{agrachev2016switching} is stated for time-optimal trajectories, but it actually holds true for extremal trajectories, since 
its proof only 
relies on the properties of the extremal flow characterized by the PMP. 
		

Upon shrinking $\mc{I}$, 
we can }
assume that $\det(H_{II}(t))\neq 0$ for every $t\in \mc{I}$. Differentiating the relation $h_I|_{\mc{I}}\equiv 0$, we find that $u(t)=H_{II}(t)^{-1}h_{0I}(t)$ holds true a.e. on $\mc{I}$.
		The differential system generated by the Hamiltonian function 
		\[
		H^0(p)=\langle p,f_0(q)\rangle+\sum_{i=1}^{2m}(H_{II}(p)^{-1}h_{0I}(p))_i\langle p,f_i(q)\rangle,\ \ p\in T^*M,\ \ q=\pi(p),
		\]
		where $(H_{II}(p)^{-1}h_{0I}(p))_i$ is the $i$-th component of $H_{II}(p)^{-1}h_{0I}(p)$, is well-defined on the set $\{ p\in T^*M\mid \Rank{(H_{II}(p))}=2m \}$. Moreover, the time-extremal triple $(q(\cdot),\lambda(\cdot),u(\cdot))$ satisfies 
		\[
		\dot\lambda(t)=\vec{H}^0(\lambda(t)),
		\]
		almost everywhere on $\mc{I}$, 
		that is, it is 
		an integral curve of $\vec{H}^0$ on $\mc{I}$. But this forces $u(\cdot)$ to be smooth on $\mc{I}$, contradicting the assumption that $t^*$ is an element of $\Sigma^{2m}$. The contradiction argument yields
		\[
			\|H_{II}(t^*)^{-1}h_{0I}(t^*)\|=1,
		\]
		and the statement follows.
	\end{proof}

As a direct consequence of Lemma~\ref{lemma:Pf} and Proposition~\ref{prop:singeven}, we deduce the following. 	
	\begin{cor}\label{cor:firstcond}
		Let $t^*\in \Sigma^{2m}\setminus \Sigma_0^{2m}$. Then, defining the symmetric $2m\times 2m$ matrix $S_H(t^*):=\adjp(H_{II})^2(t^*)$, one has
		\be\label{eq:firstcond}
			\langle S_H(t^*)h_{0I}(t^*), h_{0I}(t^*) \rangle+\det(H_{II}(t^*))=0.	
		\ee
		In particular, $\langle S_H(t^*)h_{0I}(t^*), h_{0I}(t^*) \rangle\neq 0$. 
	\end{cor}
	
	\begin{defi}\label{defi:bracketmax}
		Define the smooth functions $\left( \phi_\ell \right)_{\ell\in \N}$ and the matrix-valued functions  
		$\left( \Phi_\ell \right)_{\ell\in \N}$ 
		on $T^*M$
by
		\begin{align}
				\phi_0(\lambda)&=\langle S_H(\lambda)h_{0I}(\lambda), h_{0I}(\lambda) \rangle+\det(H_{II}(\lambda)),\label{eq:phi0}\\
			\Phi_0(\lambda)&=\left(\begin{array}{cc}  h_{0I}(\lambda) & -H_{II}(\lambda) \\ \{h_0,\phi_0\}(\lambda) &  \{h_I,\phi_0\}(\lambda)^T  \end{array}\right)\in 
			M_{2m+1}(\R),\nonumber
		\end{align}
and, inductively with respect to $\ell\ge 0$, 
		\begin{equation}\label{eq:phij}
\phi_{\ell+1}(\lambda)=\det(\Phi_\ell(\lambda)),\quad			\Phi_{\ell+1}(\lambda)=\left(\begin{array}{cc}  h_{0I}(\lambda) & -H_{II}(\lambda) \\ \{h_0,\phi_{\ell+1}\}(\lambda) &  \{h_I,\phi_{\ell+1}\}(\lambda)^T  \end{array}\right)\in 
					M_{2m+1}(\R).
		\end{equation}
	\end{defi}
	
\begin{remark}\label{rem:polynom}
By Point ii) of Lemma~\ref{lemma:Pf}, we see that $\phi_0$ in \eqref{eq:phi0} is a polynomial function in the elements $h_{ik}$ for $i\in \{0,\dots,2m\}$ and $k\in I$. Moreover, we deduce inductively that all the functions $(\phi_\ell)_{\ell\in \N}$ are polynomial functions in the elements $\mathrm{ad}_{h_{i_1}}\circ \cdots \circ \mathrm{ad}_{h_{i_\nu}}(h_{jk})(\lambda)$ for $\nu\in \N$ and $i_1,\dots,i_\nu,j,k\in\{0,\dots,2m\}$.
\end{remark}
It is useful to make the following observation on the structure of the constraint $\phi_\ell(\lambda)=0$. Its proof can be obtained by an easy inductive argument. 

\begin{lemma}\label{lem:ouf0}
Let $\ell\in \N$ and $\lambda \in T^*M$. Then  
		\begin{equation}\label{eq:firstdet}
\phi_\ell(\lambda)=			\mathrm{ad}_{h_0}^\ell(\phi_0)(\lambda)\det(H_{II}(\lambda))^\ell+B_\ell(\lambda),
		\end{equation}
		where $B_\ell(\lambda)$ is 
		the evaluation of a polynomial depending only on $\ell$ at 
		a point whose coordinates are 
		$h_{ik}(\lambda)$ 
for $i\in \{0,\dots,2m\}$ and $k\in I$, 
and $\mathrm{ad}_{h_{i_1}}\circ \cdots \circ \mathrm{ad}_{h_{i_\nu}}(\phi_0)(\lambda)$ for $0\le \nu\le \ell$ and $i_1,\dots,i_\nu\in\{0,\dots,2m\}$, with the property that if $\nu=\ell$ then $(i_1,\dots, i_\nu)\ne (0,\dots,0)$.
\end{lemma}

The following result illustrates the relation between the functions $\phi_\ell$ and the Fuller order of the set $\Sigma^{2m}$.
	
		\begin{prop}\label{prop:accmaxrank}
		Let $\ell\in\N$ and $t^*\in \Sigma^{2m}\setminus \bigcup_{j=0}^\ell\Sigma^	{2m}_j$. Then
		$\phi_j(\lambda(t^*))=0$ for every $j=0,\dots,\ell$. 
	\end{prop}

	\begin{proof}
	First notice that, since $\Sigma^{2m}$ is relatively open in $\Sigma$, 
	one has $\Sigma^{2m}_j=\Sigma^{2m}\cap\Sigma_j$ for every $j\ge 0$.

		We proceed by induction, observing that the case $\ell=0$ follows
		from Corollary~\ref{cor:firstcond}.

		Assume the conclusion to be true for some integer $\ell\ge 0$, and let us establish it for $\ell+1$. Pick $t^*\in \Sigma^{2m}\setminus \bigcup_{j=0}^{\ell+1}\Sigma^{2m}_j$ and  a sequence $(t_w)_{w\in\N}\subset  \Sigma^{2m}\setminus \bigcup_{j=0}^\ell\Sigma^{2m}_j$ converging to $t^*$. 
		The inductive step yields that 
		$\phi_j(t_w)=0$ for $j=0,\dots,\ell$ and $w\in \N$.	
		The equalities $\phi_j(t^*)=0$, $j=0,\dots,\ell$, follow by continuity, and we are left to prove that $\phi_{\ell+1}(t^*)=0$. 
		Lemma~\ref{lemma:diff} applies both to $\varphi=\phi_\ell$ and $\varphi=h_j$, $j\in I$, and allows to conclude that there exists $u^*\in\ov{B}_1^{2m}$ such that
		\be
		\Phi_{\ell+1}(\lambda(t^*))
	\begin{pmatrix}1\\ u^*\end{pmatrix}	
=0,
		\ee
where $\Phi_{\ell+1}$ is defined as in \eqref{eq:phij}.
	Hence, $\phi_{\ell+1}(\lambda(t^*))=\det(\Phi_{\ell+1}(\lambda(t^*)))=0$. 
	\end{proof}

In the next lemma, using the fact that the conditions $\phi_\ell=0$ define independent constraints on the jets, we deduce from Proposition~\ref{prop:accmaxrank}
and Lemma~\ref{lem:TL} that the set $\Sigma^{2m}$ has Fuller order at most $2n-1$. 

	\begin{lemma}\label{lemma:dominantzero}
	There exists an open and dense set $\mc{V}_{2m}\subset \Vect(M)_0^{2m+1}$ 
	such that, for every $\mathbf{f}=(f_0,\dots,f_{2m})\in \mc{V}_{2m}$ and every extremal triple $(q(\cdot),\lambda(\cdot),u(\cdot))$ of \eqref{eq:contrsys},
	\begin{equation}\label{eq:decomp2m}
		\Sigma^{2m}=\bigcup_{j=0}^{2n-1}\Sigma_j^{2m}.
	\end{equation}
	\end{lemma}

\begin{proof}
The proof of the lemma follows a classical strategy found, e.g., in \cite{BonKup}. 
Let us construct the 
set 
$\widehat{\mc{B}}\subset J^{2n+1}_{2m+1}TM\times_M T^*M$ 
by
\be\begin{aligned}
			\widehat{\mc{B}}=\bigg\{& \big(j^{2n+1}_{q}(\mathbf{f}),\lambda\big)
\,\bigg|(q,\lambda)\in T^*M,\;\mathbf{f}=(f_0,\dots,f_{2m})\in \Vect(M)_0^{2m+1},\\
						       &\det(H_{II}(\lambda))\ne0,\;\phi_0(\lambda)=\dots=
				\phi_{2n-1}(\lambda)=0\bigg\},
		\end{aligned} 
\ee
where 
$\phi_0,\dots,\phi_{2n-1}$ are defined in \eqref{eq:phi0} and \eqref{eq:phij}. 
We denote then by $\mc{B}$ the canonical projection of $\widehat{\mc{B}}$ onto $J^{2n+1}_{2m+1}TM$.
Similarly, for $q\in M$, we define $\widehat{\mc{B}}_q\subset J^{2n+1}_{2m+1,q}TM\times T^*_qM$ by 
\be
	\widehat{\mc{B}}_q:=\widehat{\mc{B}}\cap J^{2n+1}_{2m+1,q}TM\times T^*_qM,
\ee
and by $\mc{B}_q$ the canonical projection of $\widehat{\mc{B}}_q$ onto $J^{2n+1}_{2m+1,q}TM$.

Notice that, for every coordinate chart $(x,U)$, $\widehat{\mc{B}}\cap J^{2n+1}_{2m+1}TU\times T^*U$ is an algebraic subset of 
$J^{2n+1}_{2m+1}TU\times T^*U$ for the coordinates  $(X_V,x,\psi)$ introduced in Section~\ref{sec:transversality}.
Hence, $\mc{B}\cap J^{2n+1}_{2m+1}TU$ is a semi-algebraic subset  of 
$J^{2n+1}_{2m+1}TU$.

We now consider the set $\mc{V}_{2m}$ of vector fields 
$\mathbf{f}\in\Vect(M)_0^{2m+1}$ verifying the following: 
for every $q\in M$, 
$j^{2n+1}_{q}(\mathbf{f})\notin \mc{B}_{q}$.  We claim that \eqref{eq:decomp2m} holds true if 
$\mathbf{f}\in \mc{V}_{2m}$. In fact, arguing by contradiction, assume that for such an $\mathbf{f}$ and an extremal triple $(q(\cdot),\lambda(\cdot),u(\cdot))$ of \eqref{eq:contrsys},
there exists $t^*\in \Sigma^{2m}\setminus \bigcup_{j=0}^{2n-1}\Sigma_j^{2m}$. 
Then, Proposition~\ref{prop:accmaxrank} implies 
that
\begin{equation}\label{eq:dellaP}
		 \big(j^{2n+1}_{q(t^*)}(\mathbf{f}),\lambda(t^*)\big)
\in \widehat{\mc{B}},
		\end{equation} 
yielding that $j^{2n+2}_{q(t^*)}(\mathbf{f})\in \mc{B}_{q(t^*)}$ and contradicting the fact that 
$\mathbf{f}\in \mc{V}_{2m}$. The claim follows.

We conclude the proof of Lemma~\ref{lemma:dominantzero} thanks to Lemma~\ref{lem:TL}, by showing that for every 
$q\in M$, the set $\mc{B}_q$ defined above 
 has codimension larger than or equal to $n+1$ in 
 $J^{2n+1}_{2m+1,q}TM$.
 
Let $q\in M$, and consider a local coordinate chart $(x,U)$ on $M$ centered at $q$. Lift this chart to a coordinate chart $\big((x,\psi),\pi^{-1}(U)\big)$ on $T^*U$ as in Remark~\ref{rem:coor}, and recall that 
$
J^{2n+1}_{2m+1,q}TM\times T^*_qM$ is isomorphic to $P(n,2n+1)^{2m+1}\times \R^n$. 
By taking into account Remark~\ref{rem:polynom},  
 the map 
 \be
 	\begin{aligned}
		E_\phi^{2n}:
		P(n,2n+1)^{2m+1}\times \R^n
		&\to \mathbb{R}^{2n},\\
		(Q,\psi)&\mapsto \big(\phi_0(\lambda_\psi),\cdots,\phi_{2n-1}(\lambda_\psi)\big),
	\end{aligned}
 \ee
is well defined. 
 Then, up to the identification of 
 $J^{2n+1}_{2m+1,q}TU\times T^*_qU$ and $P(n,2n+1)^{2m+1}\times \R^n$,  $\widehat{\mc{B}}_q=\{(Q,\psi)\in (E_\phi^{2n})^{-1}(0)\mid \det(H_{II}(\lambda_\psi))\ne 0\}$.

In order to prove that $\mc{B}_q$ has codimension larger than or equal to $n+1$ we first show  
that $\widehat{\mc{B}}_q$ has codimension $2n$ by proving that $E_\phi^{2n}$ is a submersion at every point of $\widehat{\mc{B}}_q$. To that purpose, we compute in local coordinates the maps $\phi_i(\lambda_\psi)$ for $0\leq i\leq 2n-1$.

Following \eqref{eq:phi0} and recalling that $S_H(\lambda)\in M_{2m}(\R)$ is symmetric, we have
\be\label{eq:phi02}
	\phi_0(\lambda)=\sum_{i,j=1}^{2m}P_{i,j}(\lambda)h_{0i}(\lambda)h_{0j}(\lambda)+R_0(\lambda),
\ee
where the $P_{i,j}(\lambda)$ and $R_0(\lambda)$ are polynomial functions in the variables
$h_{st}(\lambda)$, with $1\leq s,t\leq 2m$, and not all the $P_{i,j}(\lambda)$ are zero. 
In local coordinates this gives
\begin{equation}\label{eq:phi00}
	\phi_0(\lambda_\psi)=\sum_{i,j=1}^{2m}P_{i,j}(\psi)\langle \psi,X_{0,i}\rangle\langle \psi,X_{0,j}\rangle+R_0(\psi),
\end{equation}
where the $P_{i,j}(\psi)$ and $R_0(\psi)$ are now polynomial functions in the variables $\langle \psi,X_{s,t}\rangle$, with $1\leq s,t\leq 2m$, and not all the $P_{i,j}(\psi)$ are zero.

From Lemma~\ref{lem:ouf0}, \eqref{eq:phi02} and an easy inductive argument, one deduces that, for $0\leq l\leq 2n-1$,
\be\label{eq:phil}
	\begin{aligned}
		\phi_l(\lambda)=\det(H_{II}(\lambda))^l\sum_{i,j=1}^{2m}P_{i,j,l}(\lambda)\left(h_{0^{l+1}i}(\lambda)h_{0j}(\lambda)+h_{0i}(\lambda)h_{0^{l+1}j}(\lambda)\right)+R_{0,l}(\lambda),
	\end{aligned}
\ee
where the $P_{i,j,l}(\lambda)$ are (not all zero) polynomial functions in the variables 
$h_{st}(\lambda)$,  $1\leq s,t\leq 2m$ and $R_{0,l}(\lambda)$ is a polynomial function in the variables
$\mathrm{ad}_{h_{i_1}}\circ \cdots \circ \mathrm{ad}_{h_{i_\nu}}(\phi_0)(\lambda)$ for $0\le \nu\le l$ and $i_1,\dots,i_\nu\in\{0,\dots,2m\}$, with the property that if $\nu=l$ then $(i_1,\dots, i_\nu)\ne (0,\dots,0)$. 
In local coordinates one deduces that, for $0\leq l\leq 2n-1$,
\be\label{eq:phill}
	\begin{aligned}
		\phi_l(\lambda_\psi)=&\det(H_{II}(\lambda_\psi))^l\sum_{i,j=1}^{2m}P_{i,j,l}(\psi)\bigg(\langle \psi,X_{0^{l+1},i}\rangle\langle \psi,X_{0,j}\rangle+\langle \psi,X_{0,i}\rangle\langle \psi,X_{0^{l+1},j}\rangle\bigg)\\+&R_{0,l}(\psi),
	\end{aligned}
\end{equation}
where the $P_{i,j,l}(\psi)$ are polynomial functions in the variables $\langle \psi,X_{s,t}\rangle$, $1\leq s,t\leq 2m$ and $R_{0,l}(\psi)$ is a polynomial function in the variables
$\langle \psi,X_{i_1\cdots i_\nu}\rangle$, for  
$0\le \nu\le l$ and $i_1,\dots,i_\nu\in\{0,\dots,2m\}$, with the property that if $\nu=l$ then 
$(i_1,\dots, i_\nu)\ne (0,\dots,0)$. From \eqref{eq:phi00} and \eqref{eq:phill}, one deduces that the map $E_\phi^{2n}$ is a submersion at every point of $\widehat{\mc{B}}_q$, since the polynomials $P_{i,j,l}$ are not all zero. 

We proved that $\widehat{\mc{B}}_q$ has codimension $2n$, from which 
it follows readily that the codimension of $\mc{B}_q$ is larger than or equal to $2n-n+1=n+1$ by projection, where the extra term $+1$ is due to the homogeneity of each of the relations $\phi_l(\lambda_\psi)=0$ with respect to $\lambda_\psi$. This concludes the proof of Lemma~\ref{lemma:dominantzero}.
	\end{proof}

	\section{
Iterated accumulations of points in $\Sigma$ with singular Goh matrix	
}\label{s:acc-notinv}

	We consider in this section the complementary case in which the Goh matrix $H_{II}$ does not have full rank. 

Let us fix $1\le a \le m$, and consider
the sets
	\begin{align}\label{eq:sigmaeta}
		\Sigma^{2(m-a)}&=\Sigma\cap \{ t\in [0,T]\mid \Rank{H_{II}(t)}=2(m-a) \},\\
				(T^*M)^{2(m-a)}&=T^*M\cap\left\{ \lambda\in T^*M\big\vert\, \Rank{H_{II}(\lambda)}=2(m-a) \right\}.
	\end{align}
	Observe that the notation is consistent with the notation $\Sigma^{2m}$ introduced in \eqref{eq:sigmaet}, which effectively corresponds to the case $a=0$.
	
 By point i) of Proposition~\ref{prop:princmin}, for every $\lambda\in (T^*M)^{2(m-a)}$ there exists a permutation matrix $P_\lambda\in M_{2m}(\R)$ such that 
	\be\label{eq:blockforma}
		P_\lambda^TH_{II}(\xi)P_\lambda=\begin{pmatrix}
		H_{II}^{2(m-a),\lambda}(\xi) & E^\lambda(\xi) \\
		-E^\lambda(\xi)^T & F^\lambda(\xi)
	\end{pmatrix}\ \ \text{for every }\xi\in T^*M,
	\ee
	where $H_{II}^{2(m-a),\lambda}:T^*M\to M_{2(m-a)}(\R)$, $E^\lambda:T^*M\to M_{2(m-a),2a}(\R)$ and  $F^\lambda:T^*M\to M_{2a}(\R)$ are matrix-valued functions, with the property that $H_{II}^{2(m-a),\lambda}(\lambda)$ is 
	of maximal rank (equal to $2(m-a)$).

	\begin{remark}\label{Jl0}
We assume the permutation 
matrix 
$P_\lambda$ to be chosen according to the following algorithmic rule: pick the subset $J^\lambda_0$ of $I$ of cardinality $2(m-a)$ such that the matrix extracted from  $H_{II}(\lambda)$ with row and column indices in $J^\lambda_0$ is invertible  and which is minimal for the lexicographic order among all the subsets of $I$ with the same property. (Subsets of $I$ of cardinality $2(m-a)$ are here identified with strings of indices of length $2(m-a)$.)
Then if $J^\lambda_0=\{j_1,\dots,j_{2(m-a)}\}$ and $I\setminus J^\lambda_0=\{\ell_1,\dots,\ell_{2a}\}$  with 
$j_1<\dots <j_{2(m-a)}$ and $\ell_1<\dots<\ell_{2a}$, pick as permutation the reordering of $1,\dots,2m$ into $j_1,\dots,j_{2(m-a)},\ell_1,\dots,\ell_{2a}$.
\end{remark}
		
	Consider the smooth vector-valued functions 
	\be\label{eq:ker}
	v_i^\lambda:
T^*M\to\R^{2m}, \quad \xi\mapsto\left(  \begin{array}{c} -\adjp( H_{II}^{2(m-a),\lambda}(\xi))E^\lambda(\xi) e_i \\ \Pf(H_{II}^{2(m-a),\lambda}(\xi))e_i  \end{array}  \right),\qquad i=1,\dots, 2a,
	\ee
	where  $e_1,\dots,e_{2a}$ denotes the canonical basis of $\R^{2a}$, with the convention that  $v_i^\lambda(\xi)=e_i$ 
	 when $a=m$.
	By point iii) of Proposition~\ref{prop:princmin}, there exists a neighborhood $O_\lambda\subset T^*M$ of $\lambda$ such that the collection $\{ v_i^\lambda(\xi)\mid 1\le i\le 2a \}$ parametrizes the kernel of $P_\lambda^T H_{II}(\xi)P_\lambda$ for every $\xi\in O_{\lambda}\cap (T^*M)^{2(m-a)}$.
	We also define for $1\le i\le 2a$, the functions
	\begin{align}\label{eq:kappa}
		\nonumber\kappa_i^\lambda:
T^*M&\to \R,\\
		\xi&\mapsto \langle P_\lambda^T h_{0I}(\xi), v_i^\lambda(\xi) \rangle,
	\end{align}
	and, finally, letting
	\begin{align}\label{eq:G}
	\nonumber	G^\lambda:
T^*M&\to \mathfrak{so}(2a),\\
		\xi&\mapsto E^\lambda(\xi)^T\adjp( H_{II}^{2(m-a),\lambda}(\xi))E^\lambda(\xi)+\Pf(H_{II}^{2(m-a),\lambda}(\xi))F^\lambda(\xi),
	\end{align}
	we list all of the $a(2a-1)$ independent entries of $G^\lambda$ as a collection of functions $g^\lambda_l:
T^*M\to \R$, for $1\le l\le a(2a-1)$.
Notice that $G^\lambda(\xi)=F^\lambda(\xi)=H_{II}(\xi)$ if $a=m$.
	
	\begin{prop}\label{prop:condSigmaa}
		Let $1\leq a\le m$ and  consider, for $1\le i\le 2a$ and $1\le l\le a(2a-1)$, the functions $\kappa^\lambda_i$ and $g^\lambda_l$ defined in \eqref{eq:kappa} and \eqref{eq:G}, respectively. 
		Consider an extremal triple $(q(\cdot),\lambda(\cdot),u(\cdot))$.
		Then the following holds true:
		\begin{itemize}
			\item [(i)] if $t\in\Sigma^{2(m-a)}$, then $g^{\lambda(t)}_l(t)=0$, $l=1,\dots,a(2a-1)$;
			\item [(ii)] if moreover $t\in\Sigma^{2(m-a)}\setminus \Sigma_0$, we also have $\kappa^{\lambda(t)}_i(t)=0$ for every $i=1,\dots,2a$.
		\end{itemize}
	\end{prop}

	\begin{proof}
		Our considerations being local, it is not restrictive to work with the Goh matrix $H_{II}$ in the block form \eqref{eq:blockforma}. The fact that for $t\in\Sigma^{2(m-a)}$ and $1\le l\le a(2a-1)$, $g^{\lambda(t)}_l(t)=0$ is the content of Point iii) of Proposition~\ref{prop:paramker}. If, in addition, $t$ is in $\Sigma^{2(m-a)}\setminus \Sigma_0$, then by definition there exists a nontrivial sequence $(t_l)_{l\in\N}\subset \Sigma_0$ that converges to $t$ 
		and yielding by \eqref{eq:Sigmainh=0} and Lemma~\ref{lemma:diff} the existence of some $u^*\in\ov{B}^{2m}_1$ such that 
		\[
		h_{0I}(t)-H_{II}(t)u^*=0.
		\]
		Since $H_{II}(t)$ is a skew-symmetric matrix, the above relation implies that
		\[
			h_{0I}(t)\in \ker(H_{II}(t))^\perp,
		\]
		whence $\kappa^{\lambda(t)}_i(t)=0$ 
		for every $1\le i\le 2a$.
	\end{proof}

	The following rather long and technical definition aims at identifying 
	sufficiently many independent functions that vanish at high order density points of $\Sigma$. 
	
	\begin{defi}\label{defi:brackets}
		  Let $\lambda\in (T^*M)^{2(m-a)}$ with $1\le a\le m$ and consider $\kappa^\lambda_1,\dots,\kappa^\lambda_{2a}:T^*M\to \R$ and $g^\lambda_1,\dots,g^\lambda_{a(2a-1)}:T^*M\to \R$  
defined as in \eqref{eq:kappa} and \eqref{eq:G}, respectively. 
		For every  $r\in \N$ consider 
$\rho_r^\lambda\in \{2(m-a),\dots,2m\}$, $J_r^\lambda\subset \{1,\dots,2m\}$, $\mu_r^\lambda:
T^*M\to \R$, $S_r^\lambda: 
T^*M\to M_{\rho_r^\lambda,2m}(\R)$, $T_r^\lambda: 
T^*M\to M_{\rho_r^\lambda+1,2m}(\R)$, and  $V_r^\lambda: 
T^*M\to M_{\rho_r^\lambda,1}(\R)$ defined inductively as follows: 	
		\begin{itemize}
		\item $\rho_0^\lambda=2(m-a)$, $\mu_0^\lambda=g^\lambda_1$, $J_0^\lambda$ is the 
		set defined in Remark~\ref{Jl0}, and
		\[S_0^\lambda(\xi)=\left(\begin{array}{cc} H_{II}^{2(m-a),\lambda}(\xi) & E^\lambda(\xi) \end{array}\right),
		\quad
		 V_0^\lambda(\xi)=\left(\begin{array}{c} h_{01}(\xi)\\ \vdots\\ h_{0\,
		 2(m-a)}(\xi)\end{array}\right).\]
		 (Here and in the following, $\{1,\dots,2(m-a)\}$  is identified with 
$J_0^\lambda$  by the permutation described in Remark~\ref{Jl0}.)
		Notice that $S_0^\lambda(\xi)$ is  the $2(m-a)\times 2m$ matrix obtained by selecting only 
		rows of the Goh matrix $H_{II}(\xi)$ with indices in $J_0^\lambda$;
\item for $r\ge 1$, define $\rho_r^\lambda$ to be the  rank of $S_r^\lambda(\lambda)$ and 
$J_r^\lambda$ to
be the subset of
$\{1,\dots,2m\}$
of cardinality $\rho_r^\lambda$ such that 
the 
matrix extracted from $S_r^\lambda(\lambda)$, 
with 
column indices in $J_r^\lambda$ is invertible,
and which is minimal for the lexicographic order among all 
subsets of 
$\{1,\dots,2m\}$
with the same property.

Let, moreover, for $r\ge 0$, 
\begin{equation}
T_r^\lambda(\xi)=\left(\begin{array}{c} S_r^\lambda(\xi) \\ \{ h_I,\mu_{r}^\lambda
\}(\xi) \end{array}\right)
\end{equation}
and notice that the rank of $T_r^\lambda(\lambda)$ is either equal to $\rho_r^\lambda$ or to $\rho_r^\lambda+1$;
\item 
 if $\Rank(T_r^\lambda(\lambda))=\rho_r^\lambda+1$,
 set 
\[S_{r+1}^\lambda(\xi)=T_r^\lambda(\xi),\quad 
V_{r+1}^\lambda(\xi)= \left(\begin{array}{c}V_r^\lambda(\xi) \\ \{h_0,\mu_{r}^\lambda
\}(\xi)
\end{array}\right).\]
Then $\rho_{r+1}^\lambda=\rho_{r}^\lambda+1$ and set $\mu_{r+1}^\lambda=\kappa^\lambda_{\rho_{r+1}^\lambda-\rho_0^\lambda}$;
\item 
 if $\Rank(T_r^\lambda(\lambda))=\rho_r^\lambda$
set
\[
S_{r+1}^\lambda(\xi)= S_r^\lambda(\xi),\quad
V_{r+1}^\lambda(\xi)= V_r^\lambda(\xi). 
\]
Then $\rho_{r+1}^\lambda=\rho_{r}^\lambda$.
 Let, moreover, 
$Z_r^\lambda(\cdot)$ be the matrix extracted from $S_r^\lambda(\cdot)$ 
with 
column indices in $J_r^\lambda$,
and define 
\begin{align}\tilde S^\lambda_r:
T^*M&\to M_{\rho_{r}^\lambda+1}(\R)\\
\xi&\mapsto\left(\begin{array}{cc} 
V_r^\lambda(\xi) & Z_r^\lambda(\xi)\\
\{h_0,\mu^\lambda_{r}\}(\xi)& \{ h_{J_r^\lambda},\mu_{r}^\lambda\}(\xi)\end{array}\right).\end{align} Set then $\mu_{r+1}^\lambda(\xi)=\det(\tilde S^\lambda_r(\xi))$ for every $\xi\in 
T^*M$.
\end{itemize}
\end{defi}

Notice once again that, by Proposition~\ref{prop:paramker}, the functions 
$\kappa_1^\lambda,\dots,\kappa_{2a}^\lambda$ and $g^\lambda_1,\dots,g^\lambda_{a(2a-1)}$ are polynomials in the elements $h_{jk}$ for $j,k\in\{0,\dots,2m\}$. Inductively, the  construction of Definition~\ref{defi:brackets} implies that all the functions $(\mu_r^\lambda)_{r\in \N}$, and the entries of the matrix-valued functions $(S_r^\lambda)_{r\in \N}$, $(T_r^\lambda)_{r\in \N}$ and $(V_r^\lambda)_{r\in \N}$ are polynomials in the elements $\mathrm{ad}_{h_{i_1}}\circ \cdots \circ \mathrm{ad}_{h_{i_\nu}}(h_{jk})$ for $\nu\in \N$ and $i_1,\dots,i_\nu,j,k\in\{0,\dots,2m\}$.

For every 
$\lambda\in \cup_{a=1}^m(T^*M)^{2(m-a)}$ the sequence $(\rho_r^\lambda)_{r\in\N}$ is nondecreasing
 and takes values in $\{0,\dots,2m\}$. 
	Hence, given any $N\in \N$, the pigeonhole principle implies that for every $\lambda$
	 there exists $r\le 2m N$ 
	 such that 
\begin{equation}\label{eq:plateau}
\rho_r^\lambda=\rho_{r+1}^\lambda=\dots=\rho_{r+N}^\lambda.
\end{equation}

Given $N\in \N$ and $\lambda\in  \cup_{a=1}^m(T^*M)^{2(m-a)}$, we define 
\[R_N(\lambda)=(\rho_0^\lambda,\dots, \rho_{(2m+1)N}^\lambda,J_0^\lambda,\dots,J_{(2m+1)N}^\lambda).\]  
We denote by $\Upsilon_N$ the range of  $R_N$ and we notice that it is of finite cardinality.

The main property justifying the above definition is the following. 

\begin{prop}\label{prop:higherorderconditions}
	Fix $N\ge 1$ and $\bar R\in \Upsilon_N$.
For $k=0,\dots,2(m+1)N$, denote by $\mu_k$ the function such that $\mu_k^{\lambda}=\mu_k$ for every $\lambda$ such that $R_N(\lambda)=\bar R$.
Let $(q(\cdot),\lambda(\cdot),u(\cdot))$ be an extremal triple of \eqref{eq:contrsys} 
 and define
	\[\mathfrak{S}^{\bar R}=\{t\in \Sigma\mid R_N(\lambda(t))=\bar R\}.\]
	Denote by 
	$\mathfrak{S}^{\bar R}_0$ the set of isolated points of $\mathfrak{S}^{\bar R}$ and, inductively, by $\mathfrak{S}^{\bar R}_j$ the set of isolated points of $\mathfrak{S}^{\bar R}\setminus(\bigcup_{i=0}^{j-1}\mathfrak{S}^{\bar R}_i)$. 
 Then, for every $k\in \{0,\dots,2(m+1)N\}$ and every
	\[
			t\in \mathfrak{S}^{\bar R}
			\setminus \left(\bigcup_{j=0}^{k} \mathfrak{S}^{\bar R}_j\right),
	\]
we have
		\[ 
			\mu_0(t)=\dots=\mu_{k}(t)=0. 
		\] 
\end{prop}

\begin{proof}
Let us first notice that $\rho_k^{\lambda}$, $J_k^{\lambda}$, $V_k^{\lambda}$  and the other matrices introduced in Definition~\ref{defi:brackets} do not depend on $\lambda$ provided that $R_N(\lambda)=\bar R$. To 
simplify the notations we then drop the index $\lambda$.

Let us prove the proposition by induction on $k$.
 For $k=0$ recall that $\mu_0=g_{1}$ and the conclusion follows from Proposition~\ref{prop:condSigmaa}.
The same argument works 
in the inductive step from $k-1$ to $k$ 
whenever $\rho_{k-1}<\rho_k$, since in this case 
$\mu_k=\kappa_{\rho_{k}-\rho_0}$.
When, instead, $\rho_{k-1}=\rho_k$, notice that by the inductive assumption and by
Lemma~\ref{lemma:diff} there exists  $u^*\in\overline{B}^{2m}_1$ such that 
$\{h_0,\mu_{j}\}+\sum_{i=1}^{2m} u^*_i\{h_i,\mu_{j}\}$ 
and $\{h_0,h_\ell\}+\sum_{i=1}^{2m} u^*_i\{h_i,h_\ell\}$ vanish at $\lambda(t)$ for every 
$j=1,\dots,k-1$ and every 
$\ell=1,\dots,2m$. In particular,
\[
 	(\begin{array}{cc}1 & u^*\end{array})\in\ker\left(\begin{array}{cc} 
V_{k-1}(t) & S_{k-1}(t)\\
\{h_0,\mu_{k-1}\}(t)& \{ h_{I},\mu_{k-1}\}(t)\end{array}\right).
 \] 
Since, moreover, the ranks of $\begin{pmatrix}S_{k-1}(t)\\ \{ h_{I},\mu_{k-1}\}(t)\end{pmatrix}$ 
and of its extracted matrix 
$\begin{pmatrix}Z_{k-1}(t)\\ \{ h_{J_{k-1}},\mu_{k-1}\}(t)\end{pmatrix}$
are equal, we deduce that there exists 
$v^*\in \R^{\rho_k}$ such that 
\[
 	(\begin{array}{cc}1 & v^*\end{array})\in\ker\left(\begin{array}{cc} 
V_{k-1}(t) & Z_{k-1}(t)\\
\{h_0,\mu_{k-1}\}(t)& \{ h_{J_{k-1}},\mu_{k-1}\}(t)\end{array}\right).
 \] 
  Thus $\det(\tilde S_k)(t)=\mu_{k}(t)=0$, proving 
  the claim.
\end{proof}

In order to study the independence of the constraints $\mu_j(\lambda)=0$ we 
investigate in the next lemma their expression.

	\begin{lemma}\label{lem:higherorderparameterizations}
	Fix $N\ge 1$ and $\bar R\in \Upsilon_N$.
For $k=0,\dots,2(m+1)N$, denote by $\rho_k$ the integer such that $\rho_k^{\lambda}=\rho_k$ for every $\lambda$ such that $R_N(\lambda)=\bar R$, and define similarly 
$\mu_k$, $J_k$, 
$Z_k$   and the other matrices introduced in Definition~\ref{defi:brackets}. 
Let $r,k\ge 0$ 
be such that $r+k\le (2m+1)N$, 
	\[
		\rho_r=\dots=\rho_{r+k},
	\]
and either $r=0$ or $\rho_{r-1}<\rho_r$. 	
Then 	
	\be \label{eq:relh0}
	 \mu_{r+j}(\xi)=\mathrm{ad}_{h_0}^{j}(\kappa_{\rho_r-\rho_0})(\xi)\det(Z_r(\xi))^j+P_{j}(\xi),\qquad \forall j\in \{0,\dots,k\},\ \xi\in T^*M,
	 \ee
where $P_{j}(\xi)$ is the evaluation of a polynomial 
depending only on $j$
at variables of the form $h_{i\ell}(\xi)$ 
with $i\in \{0,\dots,2m\}$ and $\ell \in J_r$, 
or $\mathrm{ad}_{h_{i_1}}\circ \cdots \circ \mathrm{ad}_{h_{i_\nu}}(\mu_\ell)(\xi)$ with $1\le \nu\le j$, $i_1,\dots,i_\nu\in\{0,\dots,2m\}$, and $\ell\in \{0,\dots, r\}$, with the property that if $\ell=r$ then $(i_1,\dots, i_\nu)\ne (0,\dots,0)$.
\end{lemma}

	\begin{proof}
Let us prove Equation~\eqref{eq:relh0} by induction on $j$. 
In the case $j=0$, by the assumption made on $r$, $\mu_{r}=\kappa_{\rho_r-\rho_0}$
and the conclusion follows. 
For $j=1,\dots,k$, $\mu_{r+j}=\det(\tilde S_{r+j-1})$, $V_{r+j}=V_r$, $Z_{r+j}=Z_r$, 
and a simple 
recursive argument allows to conclude.
	\end{proof}
	
Using the properties of the functions $\mu_j$ obtained in the last two results, we are able to prove the following lemma on the Fuller order of the set $\mathfrak{S}^{\bar R}$ introduced in the statement of Proposition~\ref{prop:higherorderconditions}. 

	\begin{lemma}\label{lemma:dominanta-new}
	Let $N\in \N$ and $\bar R\in \Upsilon_N$. Assume that $N\ge 2n$. 
	Then there exists an open and dense set $\mc{V}_{\bar R}\subset\Vect(M)_0^{2m+1}$ such that, for every $(f_0,\dots, f_{2m})\in \mc{V}_{\bar R}$,  for every extremal triple $(q(\cdot),\lambda(\cdot),u(\cdot))$ of \eqref{eq:contrsys}, 
	$\mathfrak{S}^{\bar R}$ is of Fuller order at most 
	$2(m+1)N$.
	\end{lemma}
	\begin{proof}
			Let us use the same notational convention for $\mu_j$, $\rho_j$ and the other objects introduced in Definition~\ref{defi:brackets} as in the statement of Lemma~\ref{lem:higherorderparameterizations}.
			Let $r\in \{0,\dots,2mN\}$ be minimal  such 
	that 
	\[
		\rho_r=\dots=\rho_{r+N}. 
	\]
(compare with formula~\eqref{eq:plateau}.)

 Reasoning as in Lemma~\ref{lemma:dominantzero}, define $\mc{B}\subset J^{(2m+1)N+2}_{2m+1}TM$ by projecting on $J^{(2m+1)N+2}_{2m+1}TM$ the set $\widehat{\mc{B}}\subset J^{(2m+1)N+2}_{2m+1}TM\times_M T^*M$ defined by 
\be
	\begin{aligned}
		\widehat{\mc{B}}=\bigg\{&\left(j^{(2m+1)N+2}_{q}(\mathbf{f}),\lambda\right)\,\bigg|(q,\lambda)\in T^*M,\;\mathbf{f}=(f_0,\dots,f_{2m})\in \Vect(M)_0^{2m+1},\\
		                                            &\det(Z_r(\lambda))\ne 0,\; \mu_r(\lambda)=\dots=\mu_{r+N
				}(\lambda)=0
				\bigg\}.
	\end{aligned}
\ee	
Moreover, for $q\in M$, we set $\widehat{\mc{B}}_q=\widehat{\mc{B}}\cap J^{(2m+1)N+2}_{2m+1,q}TM\times T^*_qM$ 
and 
$\mc{B}_q=\mc{B}\cap J^{(2m+1)N+2}_{2m+1,q}TM$.

We define the open set $\mc{V}_{\bar{R}}$ as the set of $\mathbf{f}\in\Vect(M)^{2m+1}_0$ with the property that, for every $q\in M$, 
 $j^{(2m+1)N+2}_{q}(\mathbf{f})\not\in \mc{B}_{q}$. 
 We claim that $\mathfrak{S}^{\bar R}$ is of Fuller order at most $2(m+1)N$ if $\mathbf{f}\in\mc{V}_{\bar{R}}$. Indeed, assume by contradiction that for $\mathbf{f}\in\mc{V}_{\bar{R}}$ and an extremal triple $(q(\cdot),\lambda(\cdot),u(\cdot))$ of \eqref{eq:contrsys} there exists $t^*\in \mathfrak{S}^{\bar R}
			\setminus \left(\bigcup_{k=0}^{2(m+1)
			N} \mathfrak{S}^{\bar R}_k\right)$. We deduce that $j^{(2m+1)N+2}_{q(t^*)}(\mathbf{f})\in \mc{B}_{q(t^*)}$ by Proposition~\ref{prop:higherorderconditions}, from which the contradiction follows.

	To conclude 
	as in 
	Lemma~\ref{lemma:dominantzero} and deduce from Lemma~\ref{lem:TL} that $\mc{V}_{\bar R}$ is dense in $\Vect(M)^{2m+1}_0$,
	it suffices to show that for every $q\in M$ the codimension 
	of $\mc{B}_q$ in 
	$J^{(2m+1)N+2}_{2m+1,q}TM$
	is larger than or equal to $n+1$.
	
	Let $q\in M$, and consider a local coordinate chart $(x,U)$ on $M$ centered at $q$. Lift this chart to a coordinate chart $\big((x,\psi),\pi^{-1}(U)\big)$ on $T^*U$ as in Section~\ref{sec:transversality}. 
	By construction, $\mc{B}\cap J^{2n+1}_{2m+1}TU$ is a semi-algebraic subset  of 
$J^{(2m+1)N+2}_{2m+1,q}TU$.

	Recall that 
$J^{(2m+1)N+2}_{2m+1,q}TM\times T^*M$ is isomorphic to $P(n,(2m+1)N+2)^{2m+1}\times \R^n$. Owing again to Remark~\ref{rem:coor}, the map
	\be
		\begin{aligned}
			\mu^{N}:P(n,(2m+1)N+2)^{2m+1}\times \R^n&\to \R^N,\\
			(Q,\psi)&\mapsto \left(\mu_r(\lambda_\psi),\dots, \mu_{r+N}(\lambda_\psi)\right)
		\end{aligned}
	\ee
	is well defined, and $\widehat{\mc{B}}_q=\{(Q,\psi)\in (\mu^{N})^{-1}(0)\mid \det(Z_r(\lambda_\psi))\ne 0\}$.
	From here, we conclude as in Lemma~\ref{lemma:dominantzero}. 
 By Proposition~\ref{prop:higherorderconditions} we have
	\be\label{eq:conditionstobe}
		\mu_{r+l}(\lambda)=\mathrm{ad}_{h_0}^l(\kappa_{\rho_r-\rho_0})(\lambda)\det(Z_r(\lambda))^l+R_l(\lambda),
	\ee
	where $R_l(\lambda)$ 
	is the evaluation of a polynomial 
depending only on $l$
at variables of the form $h_{i\ell}(\lambda)$ 
with $i\in \{0,\dots,2m\}$ and $\ell \in J_r$, 
or $\mathrm{ad}_{h_{i_1}}\circ \cdots \circ \mathrm{ad}_{h_{i_\nu}}(\mu_\ell)(\lambda)$ with $1\le \nu\le l$, $i_1,\dots,i_\nu\in\{0,\dots,2m\}$, and $\ell\in \{0,\dots, r\}$, with the property that if $\ell=r$ then $(i_1,\dots, i_\nu)\ne (0,\dots,0)$.
A routine computation of \eqref{eq:conditionstobe} in local coordinates $\big((X_{i,j})_{i,j=0}^{2m}, (\psi_r)_{r=1}^n\big)$ allows to conclude that the map $\mu^N$ is a submersion at every point of $\widehat{\mc{B}}_q$, whence we conclude that the codimension of $\mc{B}_q$ 
is greater than or equal to $N-n+1\ge 2n-n+1=n+1$, where again the $+1$ follows by the homogeneity of the relations $\mu_{r}(\lambda_\psi)=\dots=\mu_{r+N}(\lambda_\psi)=0$ with respect to $\lambda_\psi$. The conclusion follows.
	\end{proof}

	\section{Proof of Theorem~\ref{thm:main}}\label{sec:tran}

\label{ref:proof}

Let $N\ge 2n$ and define $\mc{U}=\mc{V}_{2m}\cap\left(\cap_{\bar R\in \Upsilon_N}\mc{V}_{\bar R}\right)$, where 
$\mc{V}_{2m}$ is as in Lemma~\ref{lemma:dominantzero}
and
the sets $\mc{V}_{\bar R}$ as in 
Lemma~\ref{lemma:dominanta-new}.

In particular, $\mc{U}$ is  open and dense in $\Vect(M)_0^{2m+1}$, and has the property that for every $(f_0,\dots,f_{2m})\in\mc{U}$, every extremal triple $(q(\cdot),\lambda(\cdot), u(\cdot))$ of \eqref{eq:contrsys}, 
$\Sigma^{2m}$ is of Fuller order at most 
$2n-1$
and, for every $\bar R\in \Upsilon_N$,
 $\mathfrak{S}^{\bar R}$
is of Fuller order at most $2(m+1)N$.

Denote by $N^*$ the cardinality of $\Upsilon_N$. {\color{blue} Notice that $N^*$ only depends on $n$ and $m$. }
Since $\Sigma=\Sigma^{2m}\cup\left(\cup_{\bar R\in \Upsilon_N}\mathfrak{S}^{\bar R}\right)$, we deduce from Corollary~\ref{cor:imbriques} that
$\Sigma$ has Fuller order at most $(2(m+1)N +1)N^*+
2n$. {\color{blue} Finally, since $m\le (n-1)/2$, we conclude 
the proof of Theorem~\ref{t:main} by taking $K=\max\{	(2(m+1)N +1)N^*+
2n\mid m=1,\dots,\lfloor (n-1)/2\rfloor
\}$.}

		
\bibliographystyle{abbrv}
\bibliography{biblio}
	 
\end{document}